\documentclass{article}
\usepackage[margin=1.35in]{geometry} 
\usepackage{graphicx} 

\pdfoutput = 1

\usepackage{amsmath,amsthm,amssymb,amsfonts, fancyhdr, comment, parskip, tikz, caption, subcaption, graphicx, xcolor}
\usepackage[colorlinks]{hyperref}
\usetikzlibrary{graphs}
\usetikzlibrary{shadings}
\usetikzlibrary{shapes.geometric}
\usetikzlibrary{arrows}
\usepackage{mathtools}
\usepackage{standalone}
\usepackage{upgreek}
\usepackage{diagbox}
\usepackage{booktabs} 
\usepackage{bm}
\usepackage{tikz-cd}
\usepackage{enumerate}
\usepackage{enumitem}
\usepackage{soul}
\usepackage{appendix}

\hypersetup{
    colorlinks=true,
    linkcolor=blue,
    citecolor=green,
    filecolor=magenta,
    urlcolor=blue,
    pdfborder={0 0 0}, 
}

\usepackage{mdframed}
\usepackage[capitalise]{cleveref}

\newtheorem{thm}{Theorem}
\newtheorem{prop}[thm]{Proposition}
\newtheorem{cor}[thm]{Corollary}

\theoremstyle{definition}
\newtheorem{defn}[thm]{Definition}
\newtheorem{prob}[thm]{Problem}

\newcommand{\rkn}[3]{R^{\textrm{KG}}_{#1}(#2, #3)}

\title{Ramsey numbers in Kneser graphs}
\author{Emily Heath\thanks{California State Polytechnic University Pomona.}
\and Grace McCourt\thanks{Iowa State University.}
\and Alex Parker\footnotemark[2]
\and Coy Schwieder\footnotemark[2]
\and Shira Zerbib\footnote{Iowa State University. Supported by NSF CAREER award no. 2336239  and Simons Foundation award no. MP-TSM-00002629.}}
\date{}

\begin{document}

\maketitle
\begin{abstract}
We define the {\em $r$-Kneser Ramsey number} $\rkn{r}{s}{t}$ as the minimum integer $n$ such that every red/blue edge-coloring of the Kneser graph $\textrm{KG}(n,r)$ contains a red $s$-clique  or a blue $t$-clique. We obtain general bounds on the numbers $\rkn{r}{s}{t}$,  and make progress on two related Ramsey-type problems, one raised by Holmsen, Hrusak, and Roldán-Pensado, and the other posted by Pálvölgyi. 
\end{abstract}
\section{Introduction}

The {\em Ramsey number} $R(s,t)$ is the minimum integer $n$ such that any red/blue edge-coloring of the complete graph  $K_n$ contains either a red $K_s$ or a blue $K_t$. Ramsey numbers and their many different variations have been extensively studied, and they  form one of the main pillars of extremal combinatorics. In this paper we consider a  variation on Ramsey numbers in which the host graph is a  Kneser graph rather than the complete graph. The {\em Kneser graph} $\textrm{KG}(n,r)$ is the graph whose vertex set is the set ${[n]\choose r}$ of all $r$-subsets of $[n]$, and two $r$-sets form an edge if they are disjoint. We define the {\em Kneser Ramsey numbers} as follows.

\begin{defn}
Let $r \ge 1$ be an integer. The {\em $r$-Kneser Ramsey number} $\rkn{r}{s}{t}$ is the minimum integer $n$ such that every red/blue edge-coloring of $\textrm{KG}(n,r)$ contains a red $K_s$ or a blue $K_t$.
\end{defn}

Note that  $\rkn{1}{s}{t} = R(s,t)$, so this definition  generalizes the notion of Ramsey numbers.

Our motivation to investigate Kneser Ramsey numbers originated with the following fascinating Ramsey-type question, which was  asked by Michael Hrusak for $\alpha = 1/2$ and then by Andreas Holmsen and Edgardo  Roldán-Pensado for all $\alpha$ \cite{HR}: 
\begin{prob}[Holmsen - Hrusak - Roldán-Pensado]\label{probHR}
    Fix $\alpha \in (0,1)$ and an integer $t>1$.  Does there exist a large enough $n$ (depending on $\alpha$ and $t$) such that in every red/blue edge-coloring of the complete graph on vertex set $V = {[n] \choose \alpha n}$ there is a monochromatic clique on vertex set $U \subset V$, with the property that
   the vertex-covering number of the hypergraph $([n],U)$ is at least $t$?
     \end{prob}
Although Problem \ref{probHR} is not about Kneser graphs, one observes that a positive answer will follow if in every red/blue edge-coloring of $\textrm{KG}(n,\alpha n)$ there is a monochromatic clique of size $t$.  That is, if $\rkn{\alpha n}{t}{t}\le n$ for some  $n$, then we get a positive answer to Problem \ref{probHR}. 

However,  Problem \ref{probHR} is interesting even when $\textrm{KG}(n,\alpha n)$ is empty. 
In fact,  Holmsen and Roldán-Pensado were most interested in the  special case when  $\alpha = k/(2k-1)$ and $t=2$, as formulated below.
\begin{prob}\label{probHRspecial}
Let $V = {[2k-1] \choose k}$. Is it true that for some large enough $k$, in every red/blue edge-coloring of the complete graph on vertex set $V$ there is a monochromatic triangle $ABC$ with $A,B,C\in V$ and $A\cap B \cap C = \emptyset$? 
\end{prob}

Perhaps a more approachable problem is the following.
\begin{prob}\label{probHRspecial2}
Determine the smallest number $\beta>1$ such that  for some large enough $k$, in every red/blue edge-coloring of the complete graph on vertex set $V={[\beta k] \choose k}$ there is a monochromatic triangle $ABC$ with $A,B,C\in V$ and $A\cap B \cap C = \emptyset$. 
\end{prob}
It is easy to see that $\beta \le 3$. Indeed, let $S_1,\dots, S_6 \in {[3k] \choose k}$ be six sets with the property that for every $i\in [3k]$, $i$ belongs to exactly two of the sets  $S_1,\dots, S_6$. (It is easy to construct such sets $S_1,\dots, S_6$:  take the six distinct sets in  two disjoint perfect matchings in the $k$-uniform hypergraph on vertex set $[3k]$.) Now consider a red/blue edge-coloring of the complete graph on vertices $S_1,\dots, S_6$. Since $R(3,3)=6$, this graph contains a monochromatic triangle $S_iS_jS_k$, and $S_i\cap S_j \cap S_k = \emptyset $ for every distinct $i,j,k \in [6]$.

In this paper we make progress on Problems \ref{probHR}-\ref{probHRspecial2}  in two different ways. First, in Section \ref{HRproblems} we improve the upper bound on $\beta$ in Problem \ref{probHRspecial2} to $7/3$.
\begin{thm}\label{thm: beta}
 Let $k\ge 12$ be an integer divisible by 6. Then   in every red/blue edge-coloring of the complete graph on vertex set $V={[\frac{7}{3}k] \choose k}$ there is a monochromatic triangle $ABC$ with $A,B,C\in V$ and $A\cap B \cap C = \emptyset$. 
\end{thm}

Second, in the same section, we give a positive answer to Problem \ref{probHR} when $\alpha<1/4$ and $t=3$.
\begin{thm}\label{thm: nalphan}
    For any rational number $\alpha$ such $0 < \alpha < \frac{1}{4}$, there exists some large enough integer $n$ such that $\alpha n$ is an integer and any red/blue edge-coloring of $\textrm{KG}(n, \alpha n)$ contains a monochromatic triangle.
\end{thm}

Note that when $ \alpha = \frac{1}{4}$, for $n=8$ there exists a red/blue edge-coloring of $\textrm{KG}(n, \alpha n)$ with no monochromatic triangle, as implied by Theorem \ref{thm: computational} (an explicit coloring is given in Appendix A). We do not know if this holds for larger $n$.

We then continue with a systematic study of Kneser Ramsey numbers for various parameters $r,s,t$. 
The following trivial upper bound on the $r$-Kneser Ramsey number follows from Ramsey's theorem: $$\rkn{r}{s}{t} \leq R(s,t)\cdot r.$$ 

Indeed,  when $n\ge R(s,t)\cdot r$, the graph $\textrm{KG}(n,r)$ contains a clique of size $R(s,t)$, and any red/blue edge-coloring of it contains a red $K_s$ or a blue $K_t$. This in particular shows the existence of the number $\rkn{r}{s}{t}$. Note also that if $n < kr$, then $\textrm{KG}(n,r)$ does not contain a clique of size $k$, and thus we have a trivial lower bound  $$\rkn{r}{s}{t} \geq r \cdot \max\{s,t\}.$$ 
 
In order to prove Theorem \ref{thm: nalphan}, we prove  an auxiliary result (Theorem \ref{thm: Alpha Cliques}) that in particular implies the following  bound on $\rkn{r}{s}{t}$.
\begin{thm}\label{cor: ramsey Kneser cliques}
    Fix $s, t \geq 3$ and suppose $R(s - 1, t) \geq R(s, t - 1)$. Set $M = R(s - 1, t)$ and $m = R(s, t - 1)$. 
  Then
    $$\rkn{r}{s}{t} \leq (M+1)r + m - 1, $$
    and in particular, $$\rkn{r}{s}{s} \leq (R(s, s - 1) + 1)r + R(s, s - 1) - 1.$$
\end{thm}

In Section \ref{sec:bounds} we establish further upper  bounds on $\rkn{r}{s}{t}$ for certain families of $s,t$. 

\begin{thm}\label{thm: upper bound s, s-1}
For $s \geq 4$,  we have 
\[\rkn{r}{s-1}{s} \leq rR(s-2,s) + 2(R(s-2,s-1)-1) + 2r. \] 
\end{thm}
\begin{thm}\label{thm: upper bound r,3,t}
    For $t \geq 5$, we have $\rkn{r}{3}{t} \leq rR(3,t-2) + 2r+2t-3.$
\end{thm}

We then derive a general lower bound, and an improved lower bound for $s=t=3$ and $r\ge 3$. 

\begin{prop}\label{thm: chromatic number bound}
    For  $r \geq 1$ and $s,t \geq 2$, we have $\rkn{r}{s}{t} \geq R(s, t) + 2r - 2.$
\end{prop}

\begin{thm}\label{thm: 3_3_lower_bound}
    For $r \geq 2$, we have $\rkn{r}{3}{3} \geq 3r + 2.$
\end{thm}

Using a computer code, we  also obtain  an exact value on $\rkn{2}{3}{3}$ and improve bounds on other Kneser Ramsey numbers. The computational proof of the next theorem is discussed in Subsection \ref{sec:exact}.
\begin{thm}\label{thm: computational}
    We have
        $$\rkn{2}{3}{3}=9,~~\rkn{3}{3}{3}\leq13, \text{~~and~~}  
        12 \leq \rkn{2}{4}{3} \leq 14.$$
\end{thm}

 Some of our results are summarized in Table 1 (for $r=2$) and Table 2 (for $r=3$). 
In each cell of the tables, the first number listed is the lower bound and the second is the upper bound. The subscripts on the numbers are the theorems in which we obtain the bound. For example, we have $16 \le \rkn{2}{3}{5} \le 23$, where the lower bound follows from Theorem \ref{thm: chromatic number bound} and the upper bound is proven in Theorem \ref{thm: upper bound r,3,t}. If there is only one number, then it is an exact value. Note that $\rkn{r}{s}{t}=\rkn{r}{t}{s}$.

\begin{table}[h!]
\centering
\caption{Bounds on $\rkn{2}{s}{t}$.}
\begin{tabular}{|c|| c| c| c| c|}
\hline
\diagbox{$s$}{$t$} & $t=3$ & $t = 4$ & $t = 5$ & $t = 6$\\
\hline\hline
$s = 3$ & 9\textsubscript{\ref{thm: computational}} & 12\textsubscript{\ref{thm: computational}}, 14\textsubscript{\ref{thm: computational}} & 16\textsubscript{\ref{thm: chromatic number bound}}, 23\textsubscript{\ref{thm: upper bound r,3,t}} & 20\textsubscript{\ref{thm: chromatic number bound}}, 31\textsubscript{\ref{thm: upper bound r,3,t}} \\
\hline
$s = 4$ &  & 20\textsubscript{\ref{thm: chromatic number bound}}, 28\textsubscript{\ref{cor: ramsey Kneser cliques}} & 27\textsubscript{\ref{thm: chromatic number bound}}, 48\textsubscript{\ref{thm: upper bound s, s-1}} & 38\textsubscript{\ref{thm: chromatic number bound}}, 69\textsubscript{\ref{cor: ramsey Kneser cliques}} \\
\hline
$s = 5$ &  &  & 45\textsubscript{\ref{thm: chromatic number bound}}, 76\textsubscript{\ref{cor: ramsey Kneser cliques}} & 60\textsubscript{\ref{thm: chromatic number bound}}, 132\textsubscript{\ref{thm: upper bound s, s-1}} \\
\hline
$s = 6$ &  &  &  & 104\textsubscript{\ref{thm: chromatic number bound}}, 262\textsubscript{\ref{cor: ramsey Kneser cliques}} \\
\hline
\end{tabular}
\end{table}

\begin{table}[h!]
\centering
\caption{Bounds on $\rkn{3}{s}{t}$.}
\begin{tabular}{|c|| c| c| c| c|}
\hline
\diagbox{$s$}{$t$} & $t=3$ & $t = 4$ & $t = 5$ & $t = 6$\\
\hline
\hline
$s = 3$ & 11\textsubscript{\ref{thm: 3_3_lower_bound}}, 13\textsubscript{\ref{thm: computational}} & 13\textsubscript{\ref{thm: chromatic number bound}}, 22\textsubscript{\ref{thm: upper bound s, s-1}} & 18\textsubscript{\ref{thm: chromatic number bound}}, 31\textsubscript{\ref{thm: upper bound r,3,t}} & 22\textsubscript{\ref{thm: chromatic number bound}}, 42\textsubscript{\ref{thm: upper bound r,3,t}} \\
\hline
$s = 4$ &  & 22\textsubscript{\ref{thm: chromatic number bound}}, 38\textsubscript{\ref{cor: ramsey Kneser cliques}} & 29\textsubscript{\ref{thm: chromatic number bound}}, 64\textsubscript{\ref{thm: upper bound s, s-1}} & 40\textsubscript{\ref{thm: chromatic number bound}}, 95\textsubscript{\ref{cor: ramsey Kneser cliques}} \\
\hline
$s = 5$ &  &  & 47\textsubscript{\ref{thm: chromatic number bound}}, 102\textsubscript{\ref{cor: ramsey Kneser cliques}} & 62\textsubscript{\ref{thm: chromatic number bound}}, 174\textsubscript{\ref{thm: upper bound s, s-1}} \\
\hline
$s = 6$ &  &  &  & 106\textsubscript{\ref{thm: chromatic number bound}}, 350\textsubscript{\ref{cor: ramsey Kneser cliques}} \\
\hline
\end{tabular}
\end{table}

\medskip

We conclude this paper with some progress on a Ramsey-type question that was posted by D\"om\"ot\"or  Pálvölgyi  on mathoverflow in 2016 \cite{pav}. 

Given integer $n\ge s\ge r$, an  {\em induced} $\textrm{KG}(s,r)$ is a   subgraph $G$ of the graph $\textrm{KG}(n,r)$ induced on a vertex set ${S\choose r}$, where $S \subset [n]$ is a set of size $s$. We say that $S$ is the {\em ground set} of $G$.  
\begin{prob}[Pálvölgyi \cite{pav}]\label{prob:Domotor}
    For $r>1$, determine the smallest $N\ge  3r$ such that in any red/blue edge-coloring of $\textrm{KG}(N,r)$ there exists either a red induced $\textrm{KG}(3r,r)$ or a blue triangle.
\end{prob}   

In Section \ref{sec:induced} we prove a lower bound, depending on Ramsey numbers, for the number $N$ in Pálvölgyi's problem. 
For integers $r,s,t$, let $R_r(s, t)$ be the minimum integer $n$ such that in any red/blue coloring of the edges of the complete $r$-uniform hypergraph on $n$ vertices $K^r_n$, there is a red $s$-clique or a blue $t$-clique. 

\begin{thm}\label{thm: R66}
    We have $N\ge R_r(3r, 3r)$. That is, for $n = R_r(3r, 3r) - 1$, there is a red/blue coloring of $\textrm{KG}(n,r)$ containing no red induced $\textrm{KG}(3r,r)$ and no blue triangle.
\end{thm}

For example, for $r=2$ we get $N\ge 102$, as $R_2(6,6) = R(6,6) \ge 102$ \cite{kal}.

Many of our proofs use a theorem of Alon, Frankl, and Lov\'asz~\cite{AFL}. 
\begin{thm}[Alon - Frankl - Lov\'asz~\cite{AFL}]\label{thm: Cliques in Kneser Partition}
    Let $n\ge  2r \ge 2$ and  $k_1 \geq k_2 \geq \dots \geq k_t \geq 2$ be integers, and suppose $n \geq k_1r + \sum_{i = 2}^t (k_i - 1)$. Let $F_1, F_2, \dots, F_t$ be a partition of $\binom{[n]}{r}$. Then, for some $i \in [t]$, $F_i$ contains $k_i$ pairwise disjoint sets.
\end{thm}

Observe that when $k_1 = k_2 = \cdots = k_t =2$, Theorem~\ref{thm: Cliques in Kneser Partition} recovers the chromatic number of the Kneser graph.

\section{The Holmsen - Hrusak - Roldán-Pensado problem}\label{HRproblems}

We start with the proof of Theorem \ref{thm: beta}, showing that $7/3$ is an upper bound on $\beta$ in Problem \ref{probHRspecial2}.

{\em Proof of Theorem \ref{thm: beta}}. 
Let $V=[3k/2] \cup T$, where $T$ is a set of $5k/6$ vertices disjoint from $[3k/2]$. Then $|V|=7k/3$. Suppose we have a red/blue edge-coloring of the complete graph $G$ on ${V \choose k}$. 

Consider the triangle $S_1S_2S_3$ in $G$, where   
$S_1 = \{1,\dots,k\}$, 
$S_2 = \{1,\dots,k/2\} \cup \{k+1,\dots, 3k/2\}$, and 
$S_3 = \{k/2+1,\dots,3k/2\}$.
Since $S_1\cap S_2 \cap S_3=\emptyset$, if $S_1S_2S_3$ is monochromatic we are done. Otherwise, it must have a red edge and a blue edge. Without loss of generality, 
$S_1S_2$ is red, and
$S_1S_3$ is blue. 
Note that $S_1 \cap S_i  \subset [k]$ for all $i$.

Let $T_1,T_2,T_3$ be pairwise disjoint subsets of $\{k+1,\dots, 3k/2\}$  of size $k/6$ each. Let $F_i = T\cup T_i$ for $i\in [3]$.

Assume first that the triangle $F_1F_2F_3$ is monochromatic. If it is monochromatic blue, then consider the induced graph of $G$ on vertex set $\{S_1,S_2,F_1,F_2,F_3\}$. 
Observe that if either $S_1$ or $S_2$ has  at least two blue edges to vertices in $\{F_1,F_2,F_3\}$, then we get a blue triangle of the form $S_iF_jF_l$, and since 
$S_i\cap F_j \cap F_l = S_i \cap T = \emptyset$, 
 we are done.
 Otherwise, both $S_1$ and $S_2$ have at least two red edges to vertices in $\{F_1,F_2,F_3\}$, so by the pigeonhole principle we get a red triangle of the form $S_1S_2F_i$. But now 
$(S_1\cap S_2 \cap F_i) \subset ([k] \cap F_i)=  \emptyset$, 
and we are done.
If the triangle $F_1F_2F_3$ is monochromatic red, the argument is symmetric when we consider  instead  the induced graph on vertex set $\{S_1,S_3,F_1,F_2,F_3\}$.

So suppose the triangle $F_1F_2F_3$ is not monochromatic. Then without loss of generality $F_1F_2$ is a red edge and $F_1F_3$ is a blue edge. 

Consider the induced graph on vertex set 
$\{H_1, \dots, H_6\}$, where $H_i = [k-1] \cup \{k+i-1\}$ for $i\in [6]$. 
Since $R(3,3)=6$, this graph has a monochromatic triangle, say $H_1H_2H_3$. Note that $H_i\cap H_j = [k-1]$ for $i\neq j$. Moreover, since  $k/2\ge 6$ we have $H_i \cap T = \emptyset$ for all  $i\in [6]$.  

If the triangle $H_1H_2H_3$ is monochromatic blue, then consider the induced graph of $G$ on vertex set $\{F_1,F_2,H_1,H_2,H_3\}$. 
Observe that if $F_1$ or $F_2$ has  at least two blue edges to vertices in $\{H_1,H_2,H_3\}$, then we get a blue triangle of the form $F_iH_jH_l$, and since 
$F_i\cap H_j \cap H_l = F_i \cap [k-1] = \emptyset$, 
we are done.
 Otherwise, both $F_1$ and $F_2$ have at least two red edges to vertices in $\{H_1,H_2,H_3\}$, so by the pigeonhole principle we get a red triangle of the form $F_1F_2H_i$. But now 
$(F_1\cap F_2 \cap H_i) \subset (T \cap H_i)=  \emptyset$, 
and again we are done.

Finally, if the triangle $H_1H_2H_3$ is monochromatic red, we have a symmetric argument by considering  instead  the induced graph on vertex set $\{F_1,F_3,H_1,H_2,H_3\}$.
\qed

Our next goal is to prove Theorem \ref{thm: nalphan}, which gives a positive answer in Problem \ref{probHR} in the case $\alpha<1/4$ and $t=3$. In order to prove the theorem, we need the following auxiliary result.

\begin{thm}\label{thm: Alpha Cliques}
    Fix $s, t \geq 3$ and suppose $R(s - 1, t) \geq R(s, t - 1)$. Set $M = R(s - 1, t)$ and $m = R(s, t - 1)$. Let $\alpha$ be a rational number  so that $0< \alpha < \frac{1}{M+1} $. Then, there exists some integer $n$ such that $\alpha n$ is an integer and $\textrm{KG}(n, \alpha n)$ contains a red $K_s$ or a blue $K_t$.
\end{thm}

\begin{proof}
    Fix an integer $r \geq 2$ and let $N = r(M+1) + m - 1$. We first show that in any red/blue coloring of $G= \textrm{KG}(N, r)$, there is a red $K_s$ or a blue $K_t$. Indeed, assume for contradiction that there is some coloring $c: E(G) \to \{R, B\}$ of $G$  with no red $K_s$ or blue $K_t$. Let $v = [N] \setminus [N - r] \in V(G)$. Then,  $\binom{[N-r]}{r} = \binom{[Mr + m - 1]}{r}$ can be partitioned as $N_R(v) \cup N_B(v)$, where $N_R(v)$ and $N_B(v)$ are the red and blue neighborhoods of $v$, respectively. Since $c$ contains no red $K_s$ and no blue $K_t$, then $G[N_R(v)]$ contains no red $K_{s - 1}$ and no blue $K_t$, and similarly, $G[N_B(v)]$ contains no red $K_s$ and no blue $K_{t - 1}$. In particular, $G[N_R(v)]$ contains no $K_M$ and $G[N_B(v)]$ contains no $K_m$. However, by Theorem~\ref{thm: Cliques in Kneser Partition}, setting $F_1 = N_R(v)$ and $F_2 = N_B(v)$, we have that either $N_R(v)$ contains a clique of size $M$ or $N_B(v)$ contains a clique of size $m$, a contradiction.

    Now, let $\alpha$ be a rational number  so that $0< \alpha < \frac{1}{M+1}$. Let 
    $r$ be a positive integer  so that $\frac{r}{(M+1)r + m - 1} \geq \alpha$ and $\frac{r}{\alpha}$ is an integer.  Let $N = r+ Mr + m - 1$ and $n =  \frac{r}{\alpha}$. By the choice of $r$ we have $n\ge N$. Therefore, $\textrm{KG}(N, r) \subseteq \textrm{KG}(n, \alpha n)$, and thus any red-blue coloring of $\textrm{KG}(n, \alpha n)$ contains a red $K_s$ or a blue $K_t$. 
\end{proof}

Note that the general bound on $\rkn{r}{s}{t}$ in Theorem \ref{cor: ramsey Kneser cliques} is proven in the first paragraph of the proof of Theorem~\ref{thm: Alpha Cliques}.

We are now ready to prove Theorem \ref{thm: nalphan}. 

\begin{proof}[Proof of Theorem \ref{thm: nalphan}]

We apply Theorem~\ref{thm: Alpha Cliques} with $s=t=3$. This gives $M = m = R(2,3)=3$. Let $\alpha$ be a rational number satisfying $0 < \alpha< \frac{1}{4}$. Then $\alpha < \frac{1}{M+1}$, and  thus  by Theorem~\ref{thm: Alpha Cliques} there exists some integer $n$ such that $\alpha n$ is an integer  and any red/blue coloring of the edges of $\textrm{KG}(n, \alpha n)$ contains a monochromatic triangle.
\end{proof}

\section{General bounds on Kneser Ramsey numbers}\label{sec:bounds}

In this section we prove Theorems \ref{thm: upper bound s, s-1}, \ref{thm:  upper bound r,3,t},  \ref{thm: 3_3_lower_bound}, and \ref{thm: computational},  and Proposition \ref{thm: chromatic number bound}. 

\subsection{Upper bounds}
The proofs of both Theorem \ref{thm: upper bound s, s-1} and Theorem \ref{thm: upper bound r,3,t} follow the same scheme, described as follows. For some integers $N,r,s,t$, let $G = \textrm{KG}(N, r)$ and let $c: E(G) \to \{R, B\}$ be a red/blue edge-coloring of $G$. Assume that $c$ has  no red $K_{s}$ or blue $K_t$, and suppose there exists a blue edge in $c$. Then (by renaming the integers in $[N]$) we may assume without loss of generality that $c(uv) = B$,  where $u = \{N, N - 1, \dots, N - r + 1\} \in V(G)$ and $v = \{N-r, N-r-1, \dots, N - 2r + 1\} \in V(G)$. Note that  $G[N_G(u)] \cap G[N_G(v)]= \textrm{KG}(N-2r, r)$. 
    
    We now partition the set of vertices ${[N-2r]\choose r}$ of $G'=\textrm{KG}(N-2r, r)$ into four sets $N_{XY}$, where $X,Y \in \{R,B\}$, where $N_{XY}$ is the set of vertices $w\in {[N-2r]\choose r}$ such that $c(wu) = X$ and $c(wv) = Y$.

    Observe that $N_{BR}$ contains no red $K_{s-1}$, for otherwise we  have  a red $K_{s}$ together with $v$, and additionally, $N_{BR}$  contains no blue $K_{t-1}$, for otherwise have a blue $K_{t}$ together with $u$. Thus $$\omega(G'[N_{BR}]) \leq R(s-1, t-1) -1$$ where $\omega(H)$ is the clique number of the graph $H$. 
    Similarly, $N_{RR} \cup N_{RB}$  contains no red $K_{s-1}$ or blue $K_t$, so $$\omega(G'[N_{RR} \cup N_{RB}]) \leq R(s-1,t)-1.$$ Finally  $N_{BB}$  contains no red $K_{s}$ and no blue $K_{t-2}$, as otherwise we can find a blue $K_t$ together with the edge $uv$. Hence $$\omega(G'[N_{BB}]) \leq R(s,t-2)-1.$$ 

Now we would like to apply Theorem \ref{thm: Cliques in Kneser Partition}  to conclude that at least one of the above three induced subgraph of $G'$ contains a large clique, which will constitute a contradiction. 
 But in order to do that, we need to compare the sizes of the Ramsey numbers $R(s-1, t-1)$, $R(s-1,t)$, and $R(s,t-2)$, which we do not know how to do in general. However, in the two special cases where $s$ and $t$ differ by 1 (Theorem \ref{thm: upper bound s, s-1}) or when $s=3$ (Theorem \ref{thm: upper bound r,3,t}) we can find the largest value among the three Ramsey numbers, which allows us to apply Theorem  \ref{thm: Cliques in Kneser Partition}.

\begin{proof}[Proof of Theorem \ref{thm: upper bound s, s-1}] 
We would like to show that for $s \geq 4$, 
\[\rkn{r}{s-1}{s} \leq rR(s-2,s) + 2(R(s-2,s-1)-1) + 2r. \]

    Let $N = rR(s-2,s) + 2(R(s-2,s-1)-1) + 2r$, $G = \textrm{KG}(N, r)$ and let $c: E(G) \to \{R, B\}$ be a red/blue edge-coloring of $G$. Assume for contradiction that $c$ has  no red $K_{s-1}$ or blue $K_s$. Since  $$|V(G)| = {N \choose r} \ge N > rR(s-2,s)=rR(s,s-2),$$ there is a clique of size $R(s,s-2)$ in $G$, so if $c$ has no blue edge it must have a red $K_s$. So we may assume some edge is colored blue. Without loss of generality, suppose $c(uv) = B$,  where $u = \{N, N - 1, \dots, N - r + 1\} \in V(G)$ and $v = \{N-r, N-r-1, \dots, N - 2r + 1\} \in V(G)$. Then  $G[N_G(u)] \cap G[N_G(v)]= \textrm{KG}(N-2r, r)$. Let $G'= \textrm{KG}(N-2r, r) $.

    As explained above, we have (for $(s,t)\leftarrow (s-1,s)$)
\begin{equation}
    \begin{split}
        \omega(G'[N_{BR}]) &\leq R(s-2, s-1) -1, \\
        \omega(G'[N_{RR} \cup N_{RB}]) &\leq R(s-2,s)-1, \\
        \omega(G'[N_{BB}]) &\leq R(s-1,s-2)-1 = R(s-2, s-1)-1.
    \end{split}
\end{equation}

Clearly, we have $R(s-2,s) \ge R(s-2,s-1)$,  and thus we can 
 invoke Theorem~\ref{thm: Cliques in Kneser Partition} with $F_1 = N_{RR} \cup N_{RB}$, $F_2 = N_{BR}$ and $F_3 = N_{BB}$, and with $k_1 = R(s-2,s),$ $k_2 = R(s-2, s-1),$  $k_3 = R(s-2,s-1)$. Since $$k_1\cdot r + (k_2 - 1) + (k_3 - 1) = N-2r,$$ Theorem \ref{thm: Cliques in Kneser Partition} entails that either $F_1$ contains a clique on $R(s-2,s)$ vertices, or $F_2$ contains a clique on $R(s-2, s-1)$ vertices, or $F_3$ contains a clique on $R(s-2, s-1)$ vertices, contradicting (1). 
\end{proof}

\begin{proof}[Proof of Theorem \ref{thm: upper bound r,3,t}]
 We now apply our proof scheme with $(s,t)\leftarrow (3,t)$.
 We would like to show that for  $t \geq 5$, 
\[\rkn{r}{3}{t} \leq rR(3,t-2) + 2r+2t-3. \] 
Let $N = rR(3,t-2) + 2r+2t-3$ and $G = \textrm{KG}(N, r)$. Let $c: E(G) \to \{R, B\}$ be a red/blue edge-coloring of $G$. Assume for contradiction that $c$ has  no red $K_{3}$ nor blue $K_t$. Like before, since  $N > rR(3,t-2)$,  if $c$ has no blue edge, then it must have  a red $K_3$. So we may assume some edge is colored blue, and without loss of generality, suppose $c(uv) = B$,  where $u$ and $v$ as before. Then  $G[N_G(u)] \cap G[N_G(v)]= \textrm{KG}(N-2r, r)$. Let $G'=\textrm{KG}(N-2r, r)$.
    
Now we have 
\begin{equation}
    \begin{split}
        \omega(G'[N_{BR}]) &\leq R(2, t-1) -1 = t-2, \\
        \omega(G'[N_{RR} \cup N_{RB}]) &\leq R(2,t)-1=t-1, \\
        \omega(G'[N_{BB}]) &\leq R(3,t-2)-1.
    \end{split}
\end{equation}

It is straightforward to see  that $R(3,t-2) > 2(t-3)$ (indeed, take two disjoint cliques on $t-3$ vertices each, and color their edges blue; color the edges between them red; then the red graph is bipartite so contains no triangles). This implies that 
 for $t\ge 5$ we have $R(3,t-2) \ge t$.
 Therefore,  applying  Theorem \ref{thm: Cliques in Kneser Partition} with $F_1 = N_{RR} \cup N_{RB}$, $k_1 = R(3,t-2)$, $F_2 = N_{BR}$, $k_2 = t$,  $F_3 = N_{BB}$, and $k_3 = t-1$, we have $$N - 2r \geq r(R(3,t-2)) + (t-1 ) + (t-1 - 1),$$ implying that for some $i\in [3]$,  $F_i$ contains a clique of size $k_i$, contradicting (2). 
\end{proof}

\subsection{Lower bounds}

For the lower bound in Proposition \ref{thm: chromatic number bound} we employ the chromatic number of Kneser graphs. It is easy to verify that $\chi(\textrm{KG}(n,r)) \ge n - 2r + 2$ by exhibiting a explicit proper coloring, and this is the bound we need for the proof. 
Lov\'asz \cite{lov} proved that $\chi(\textrm{KG}(n,r)) = n - 2r + 2$, which was one of the earliest applications of topological methods to combinatorics.

\begin{proof}[Proof of Proposition~\ref{thm: chromatic number bound}]
 For   $r \geq 1$ and $s,t \geq 2$, we wish to show $\rkn{r}{s}{t} \geq R(s, t) + 2r - 2.$
Let $n = \rkn{r}{s}{t}$ and set $G= \textrm{KG}(n,r)$. Assume for contradiction that $R(s,t) > n-2r+2 = \chi(G)$ and let $c:V(G) \to [n - 2r + 2]$ be a proper vertex-coloring of $G$. Consider the complete graph $G'$ on vertex set  $[n - 2r + 2]$ and let $c':E(G') \to \{R, B\}$ be a red/blue edge-coloring of $G'$ avoiding a red copy of $K_s$ and a blue copy of $K_t$. Note that $c'$ exists because of our negation assumption.

    Now, define a red/blue edge-coloring of $f:E(G)\to \{R, B\}$ by
     $$f(uv)=c'(c(u), c(v))$$ for  $uv \in E(G)$. 
     We claim that $f$ has no red copy of $K_s$ or blue copy of $K_t$. Suppose there is a red copy of $K_s$ with vertices $\{v_1, v_2, \dots, v_s\}$. Then for $i \neq j$ we have  $$f(v_iv_j) = c'(c(v_i), c(v_j)) = R.$$ Observe that when   $i \neq j$, $c(v_i) \neq c(v_j)$ because $c$ is a proper vertex-coloring of $G$ and $v_iv_j \in E(G)$. This implies that the subgraph of $G'$ induced on the vertices $\{c(v_1), c(v_2), \dots, c(v_s)\}$ is a red copy of $K_s$, a contradiction to our choice of $c'$. A similar argument shows that $f$ has no blue copy of $K_t$.
\end{proof}

\begin{proof}[Proof of Theorem \ref{thm: 3_3_lower_bound}]
    We wish to show there exists a red/blue coloring $c$ of $G = \textrm{KG}(3r + 1, r)$ with no monochromatic triangle. To this end, we partition $V(G)$ as follows:
    \begin{align*}
        X &= \{x \in V(G) : 3r + 1 \in x\}, \\
        Y &= \{x \in V(G) \setminus X : 3r \in x\}, \\
        Z &= V(G) - Y - X \simeq V(\textrm{KG}(3r - 1, r)).
    \end{align*}
    Observe that  the induced subgraph $G[Z]$ of $G$ on vertex set $Z$ is triangle-free. Also, $X$ and $Y$ are independent sets with no triangles contained in $G[X \cup Y]$. Therefore, the only triangles will have either  two vertices from $Z$ and one from $X \cup Y$ or  one vertex from each of $X, Y, Z$. Setting $c(E(G[Z])) = B$ and $c(E(X, Y)) = B$ with all other edges colored $R$ (see Figure \ref{fig: 3_3_lower_bound}), we obtain a red/blue coloring with no monochromatic triangles. 
\end{proof}

    \begin{figure}[h!]
    \centering
        \begin{tikzpicture}[every node/.style={font=\small}]
        \node[draw, ellipse, minimum width=1cm, minimum height=2cm, label = left: $X$] (X) at (-2, 1) {};
        \node[draw, ellipse, minimum width=1cm, minimum height=2cm, label = right: $Y$] (Y) at (2, 1) {};
        \node[draw, ellipse, minimum width=2cm, minimum height=1cm, label = below: $Z$] (Z) at (0,-2) {};
    
        \foreach \yX in {0.2, 0.6, 1.0, 1.4, 1.8}{
            \foreach \yY in {0.2, 0.6, 1.0, 1.4, 1.8}{
                \draw[thin, blue] (-1.7,\yX) -- (7.7-6,\yY);
            }
        }
        
        \foreach \yX in {0.2, 0.6, 1.0, 1.4, 1.8}{
            \foreach \xZ in {-0.4, 0, 0.4}{
                \draw[thin, red] (-1.7,\yX) -- (\xZ, -2);
            }
        }

        \foreach \yY in {0.2, 0.6, 1.0, 1.4, 1.8}{
            \foreach \xZ in {-0.4, 0, 0.4}{
                \draw[thin, red] (1.7,\yY) -- (\xZ, -2);
            }
        }
    
        \foreach \xA in {-0.5, -0.3, -0.1, 0.1, 0.3, 0.5}{
            \foreach \xB in {-0.5, -0.3, -0.1, 0.1, 0.3, 0.5}{
                \draw[thin, blue] (\xA, -2.25) -- (\xB, -1.75);
            }
        }
        
    \end{tikzpicture}
    \caption{The coloring in \cref{thm: 3_3_lower_bound}}
    \label{fig: 3_3_lower_bound}
\end{figure}

\subsection{Computationally-obtained  bounds}\label{sec:exact}

To obtain the upper and lower bounds in Theorem \ref{thm: computational} we wrote a program in Python, which for   (small enough) specified parameters $n,r,s,t$, runs over all possible red/blue edge-colorings of $\textrm{KG}(n,r)$ and either provides as output a coloring that contains no red $K_s$ and no blue $K_t$ (which we call ``good"), or declares that a good coloring does not exists. 

We do this by reformulating the question as a satisfiability problem. For example, for $\rkn{r}{3}{3}$, we look for a red/blue edge-coloring $c$ of $G = \textrm{KG}(n, r)$ with no monochromatic triangle. Suppose $E(G) = \{e_1, \dots, e_{|E(G)|}\}$. For $1 \leq i \leq |E(G)|$, we may assign a boolean variable $v_i$ corresponding to $e_i$, where 
\[
v_i = 
\begin{cases}
    1 \text{ (True)} & \text{if } c(e_i) = R, \\
    0 \text{ (False)} & \text{if } c(e_i) = B.
\end{cases}
\]

Then, for a triangle $T = \{e_1, e_2, e_3\}$ of $G$, $T$ is monochromatic if and only if the following boolean clause is true:
$$(v_1 \land v_2 \land v_3) \lor (\neg v_1 \land \neg v_2 \land \neg v_3).$$

Negating this, we get that $T$ is not monochromatic if and only if the following boolean clause, $B_T$, is true:
$$B_T=(v_1 \lor v_2 \lor v_3) \land (\neg v_1 \lor \neg v_2 \lor \neg v_3).$$

Let $\{T_1, \dots, T_{t}\}$ be the set of all triangles in $G$. Then, there is a good edge-coloring of $G$ if and only if there exists an assignment  $(v_1, \dots, v_{|E(G)|}) \in \{0,1\}^{|E(G)|}$ such that the   clause $B=\bigwedge_{i = 1}^t B_{T_i}$ is satisfiable (True).

Our program builds the clause $B$ and then checks its  satisfiability.

The proof of Theorem \ref{thm: computational} for $\rkn{2}{3}{3}$ follows because our program found a good coloring of $\textrm{KG}(8,2)$, but declared that the graph $\textrm{KG}(9,2)$ does not have a good coloring. The bounds for the other Kneser Ramsey numbers in Theorem \ref{thm: computational} were obtained in similar ways.

The explicit colorings of $\textrm{KG}(8,2)$ with no monochromatic $K_3$, and of $\textrm{KG}(11,2)$ with no red $K_4$ or blue $K_3$ are given in Appendix A. Our code is given in Appendix B.

\section{Pálvölgyi's problem}\label{sec:induced}

\subsection{Induced Kneser graphs }
We start with a discussion on induced graphs of the Kneser graph. Recall that given integer $n\ge s\ge r$, we defined an {\em induced} $\textrm{KG}(s,r)$ as a subgraph $H$ of the graph $\textrm{KG}(n,r)$ induced on a vertex set ${S\choose r}$, where $S \subset [n]$ is a ground set of size $s$. 

We first show that  
 every induced graph $H$ of $\textrm{KG}(n,r)$  {\em isomorphic} to $\textrm{KG}(2r,r)$ is, in fact, an induced $\textrm{KG}(2r,r)$: 

\begin{prop}\label{prop: set systems}
   Suppose  $H$ is an induced subgraph of $G=\textrm{KG}(n,r)$ isomorphic to $\textrm{KG}(2r,r)$. Then $H$ is an induced $\textrm{KG}(2r,r)$,  that is, there exists a ground set $S\subset [n]$ of size $2r$ such that $H=G[{S\choose r}]$.
\end{prop}

 The proof will follow  from the set-system theorem of Bollob\'as.
\begin{thm}[Bollob\'as \cite{bol}]\label{lem: two_families}
    Let $(A_1, B_1), \dots, (A_m, B_m)$ be  pairs of sets with $|A_i| = a$ and $|B_i| = b$ for every $i$.  Suppose that 
   $A_i \cap B_i = \emptyset$ for $1 \leq i \leq m$, and
    $A_i \cap B_j \neq \emptyset$ for $i \neq j$.
    Then $m \leq \binom{a + b}{a}$.  Furthermore, if $m = \binom{a + b}{a}$ then there is some set $S$ of cardinality $a + b$ such that the sets $A_i$ are all the subsets of $S$ of size $a$, and $B_i = S \setminus A_i$ for all $i$.
\end{thm}

\begin{proof}[Proof of Proposition \ref{prop: set systems}]
 Observe that $\textrm{KG}(2r,r)$ is a matching of size $\frac{1}{2} \binom{2r}{r}$. 
 We claim that every induced matching in $\textrm{KG}(n,r)$ of size  $\frac{1}{2} \binom{2r}{r}$ has vertex set $\binom{S}{r}$, for some $S\subset [n]$ and $ |S|= 2r$.
Indeed, let $M = \{ (X_i, Y_i) \colon 1 \leq i \leq m \}$ be an induced matching of $\textrm{KG}(n,r)$. Then   $X_i \cap Y_i = \emptyset$ for all $i\in [m]$, and $X_i \cap Y_j \neq \emptyset$ for all $i \neq j$. 

Set 
\[ A_i=  \left\{
\begin{array}{ll}
      X_i & 1 \leq i \leq m \\
      Y_{i-m} & m+1 \leq i \leq 2m, \\
\end{array} 
\right. \]
and 
\[ B_i=  \left\{
\begin{array}{ll}
      Y_i & 1 \leq i \leq m \\
      X_{i-m} & m+1 \leq i \leq 2m. \\
\end{array} 
\right. 
\]

 Then the set pair system $\{(A_i, B_i) \mid 1 \leq i \leq 2m\}$ satisfies the conditions of Theorem \ref{lem: two_families}. Therefore, by the theorem,
    $$2 |M| = 2m \leq \binom{2r}{r},   $$ and moreover, equality is achieved when the sets $A_i$ are all the $r$ subsets of the set  $S=A_1\cup B_1$, where $B_i = S \setminus A_i$ for all $i$.
    This implies that $H$ is a induced $KG(2r,r)$ on ground set $S$, as claimed.  
\end{proof}

We return now to Pálvölgyi's problem. In Problem~\ref{prob:Domotor} one would like to find a red/blue edge-coloring of $G=\textrm{KG}(n,r)$ with no red  induced   $\textrm{KG}(3r,r)$  and no blue triangle. 
We first show the existence of a red/blue edge-coloring of $G$ with no  monochromatic induced Kneser graphs $\textrm{KG}(s,r)$ for all  $n\ge s\ge 2r$.

\begin{prop}\label{prop: interlacing}
    Let $n\ge 2r\ge 4$ be integers. Then there exists a red/blue edge-coloring of $\textrm{KG}(n,r)$ with no monochromatic  induced   $\textrm{KG}(s,r)$ for all  $n\ge s\ge 2r$. 
\end{prop} 
\begin{proof}
 We say that an edge $XY \in \textrm{KG}(n,r)$ is interlacing, where $X,Y \in {[n]\choose r}$ are disjoint sets, if  there exist $x,x'\in X$ and  $y,y'\in Y$ such that either $y<x<y'$ or  $x<y<x'$. Otherwise, we say that $XY$ is non-interlacing. Now color the interlacing edges of $\textrm{KG}(n,r)$ red, and the non-interlacing edges blue. 
 
 We claim that in this coloring there is no monochromatic $\textrm{KG}(s,r)$ for any  $n\ge s\ge 2r$. Indeed, 
suppose  $H$ is an induced  $\textrm{KG}(s,r)$ on ground set $S=\{z_1< z_2 <\dots <z_s\}\subset [n]$ where $s\ge 2r\ge 4$. 
 Then the edge $XY \in H$, where $X=\{z_1, \dots z_r\}$  and $Y=\{z_{r+1}, \dots z_{2r}\}$ is non-interlacing, while the edge $X'Y' \in H$, where $X'=\{z_2, \dots z_{r+1}\}$  and $Y'=\{z_{1},z_{r+2} \dots z_{2r}\}$ is interlacing, showing that $H$ is not monochromatic.  
\end{proof}

The proof shows that an induced $\textrm{KG}(s,r)$ is not monochromatic because it contains an induced $\textrm{KG}(2r,r)$, which is already not monochromatic under the edge coloring in the proof. 
Therefore, by Proposition \ref{prop: set systems}  we can state Proposition \ref{prop: interlacing} in a slightly 
 stronger manner: 
 \begin{cor}\label{prop: interlacing2}
    Let $n\ge 2r\ge 4$ be integers. There exists a red/blue edge-coloring of $G=\textrm{KG}(n,r)$ such that for all  $n\ge s\ge 2r$, if  $H$ is an induced subgraph of $G$ isomorphic to 
 $\textrm{KG}(s,r)$, then $H$ is not monochromatic. 
\end{cor}

Note that Problem \ref{prob:Domotor} is not resolved by Corollary \ref{prop: interlacing2} because a triangle is an induced Kneser graph only if $r=1$. In the next subsection we prove Theorem \ref{thm: R66}, which gives a lower bound for Problem \ref{prob:Domotor}.

\subsection{Proof of Theorem~\ref{thm: R66}}
Set $n = R_r(3r, 3r) - 1$, where  $R_r(3r, 3r)$ is the minimum integer $m$ such that in any red/blue coloring of the edges of the complete $r$-uniform hypergraph on $m$ vertices $K^r_m$ there is a monochromatic $K^r_{3r}$. 
We would like to show that there is a red/blue coloring $c$ of $G=\textrm{KG}(n,r)$ containing no red induced $\textrm{KG}(3r,r)$ and no blue triangle.

 Let $c'$ be a red/blue edge coloring of the graph $G' = K^r_{n}$ containing no monochromatic $K^r_{3r}$ ($c'$ exists by the choice of $n$). We partition $E(G') = V(G)$ into two sets $P_1, P_2$ as follows: for $e \in E(G')$, if $c'(e) = R$,  then $e\in P_1$. Otherwise,  $e\in P_2$.

Observe that by the choice of $c'$, neither $G[P_1]$ nor $G[P_2]$ contain an induced copy of $\textrm{KG}(3r,r)$. Indeed, if  $G[P_i]$ contains an induced copy of $\textrm{KG}(3r,r)$, then there exists $S\subset [n]$ of size $3r$ such that ${S\choose r} \subset P_i$. But this implies that $c'$ contains a monochromatic copy of $K^r_{3r}$ on vertex set $S$.

We now define $c:E(G) \to \{R,B\}$ as follows. 
For any edge $e \in G[P_R]$ or $e\in G[P_B]$, set $c(e) = R$. Otherwise, set $c(e) = B$. Since the blue subgraph is a bipartite graph, it is triangle-free. Hence $c$ is a red/blue edge coloring of $G$ containing no red induced $\textrm{KG}(3r,r)$ and no blue triangle. 
\qed

\section{Concluding remark}
The {\em multicolor Ramsey number} $R(s_1,s_2, \dots, s_t)$ is the minimum integer $n$ such that in any $t$-coloring of the edges of $K_n$, there exists a color-$i$ copy of $K_{s_i}$ for some $i \in [t]$. We can define the multicolor $r$-Kneser Ramsey number $\rkn{r}{s_1, \dots}{s_t}$ analogously.

While we do not consider multicolor Kneser Ramsey numbers in depth in this paper, by similar arguments as in the proof of Theorem~\ref{thm: Alpha Cliques}, we  obtain an upper bound on these numbers. 

\begin{thm}
    Fix $s_1, s_2, \dots, s_t \geq 3$ and for $1 \leq i \leq t$, set $m_i = R(s_1, s_2, \dots, s_i - 1, \dots, s_t)$. Suppose $m_1 \geq m_j$ for all $j \geq 2$. Then,
    $$R^{\textrm{KG}}_{r}(s_1, \dots, s_t) \leq (m_1 + 1)r + \sum_{i = 2}^t (m_i - 1).$$
\end{thm}

\bigskip

\appendix

\section{Appendix: Explicit colorings.}\label{Apen: colorings}

\subsection{A good coloring of $\textrm{KG}(8,2)$}
Here we provide a coloring of $\textrm{KG}(8,2)$ with no monochromatic $K_3$. 
This shows $\rkn{2}{3}{3} \ge 9$.

We can describe part of this coloring as follows. First, we let $A = \binom{[3]}{2}$, $B = \binom{[8] - [3]}{2}$, and $C = V(G) - (A \cup B)$. Then all the edges from $A$ to $C$ as well as all the edges within $B$ are red, and all the edges from $A$ to $B$ are blue.

The red edges are:

Edges from $A$ to $C$:

$((1, 2), (3, 4)), ((1, 2), (3, 5)), ((1, 2), (3, 6)), ((1, 2), (3, 7)),((1, 2), (3, 8)), ((1, 3), (2, 4)),\\ 
((1, 3), (2, 5)), ((1, 3), (2, 6)),((1, 3), (2, 7)), ((1, 3), (2, 8)), ((1, 4), (2, 3)), ((1, 5), (2, 3)),\\
((1, 6), (2, 3)), ((1, 7), (2, 3)),((1, 8), (2, 3))$

Edges from $B$ to $C$:

$((1, 4), (5, 6)),((1, 4), (6, 8)),  ((1, 5), (6, 7)), ((1, 5), (6, 8)), ((1, 5), (7, 8)), ((1, 6), (4, 5)), \\
((1, 6), (5, 7)),((1, 6), (5, 8)), ((1, 7), (4, 8)), ((1, 7), (5, 8)), ((1, 8), (4, 6)), ((1, 8), (4, 7)), \\
((1, 8), (6, 7)),
((2, 4), (5, 6)), ((2, 4), (6, 8)), ((2, 5), (6, 7)), ((2, 5), (6, 8)),((2, 5), (7, 8)), \\
((2, 6), (4, 5)),((2, 6), (5, 7)), ((2, 6), (5, 8)), ((2, 7), (4, 5)), ((2, 7), (4, 8)),
((2, 7), (5, 8)), \\
((2, 8), (4, 5)), ((2, 8), (4, 6)), ((2, 8), (4, 7)),
((3, 4), (5, 6)), ((3, 4), (6, 8)), ((3, 5), (6, 7)), \\
((3, 5), (6, 8)),((3, 5), (7, 8)), ((3, 6), (4, 5)),
((3, 6), (5, 7)), ((3, 6), (5, 8)),((3, 7), (4, 8)), \\ ((3, 7), (5, 8)), ((3, 8), (4, 6)), ((3, 8), (4, 7)),
((3, 8), (6, 7))
$

Edges within $B$:

$((4, 5), (6, 7)), ((4, 5), (6, 8)), ((4, 5), (7, 8)),((4, 6), (5, 7)), ((4, 6), (5, 8)),
((4, 6), (7, 8)), \\ ((4, 7), (5, 6)),((4, 7), (5, 8)), ((4, 7), (6, 8)), ((4, 8), (5, 6)), ((4, 8), (5, 7)),
((4, 8), (6, 7)), \\ ((5, 6), (7, 8)), ((5, 7), (6, 8)), ((5, 8), (6, 7))$

Edges within $C$:

$
  ((1, 4), (2, 6)),
((1, 4), (2, 8)), ((1, 4), (3, 6)), ((1, 4), (3, 8)),  
((1, 5), (2, 6)), ((1, 5), (2, 7)), \\
((1, 5), (3, 6)), ((1, 5), (3, 7)),
  ((1, 6), (2, 4)), ((1, 6), (2, 5)),((1, 6), (3, 4)), ((1, 6), (3, 5)),\\
  ((1, 7), (2, 5)), ((1, 7), (2, 8)),
((1, 7), (3, 5)), ((1, 7), (3, 8)),  ((1, 8), (2, 4)),
((1, 8), (2, 7)), \\
((1, 8), (3, 4)),((1, 8), (3, 7)), 
((2, 4), (3, 6)), ((2, 4), (3, 8)), ((2, 5), (3, 6)), ((2, 5), (3, 7)),\\
((2, 6), (3, 4)), ((2, 6), (3, 5)), 
 ((2, 7), (3, 5)), ((2, 7), (3, 8)),
 ((2, 8), (3, 4)),((2, 8), (3, 7)) $

The remaining edges are blue. \\

\subsection{A good coloring of $\textrm{KG}(11,2)$}
Below is a coloring of $\textrm{KG}(11,2)$ with no red $K_4$ or blue $K_3$. This shows $\rkn{2}{4}{3} \ge 12$.

The red edges are:

$((1, 2), (3, 5)), ((1, 2), (3, 9)), ((1, 2), (3, 10)), ((1, 2), (4, 5)), ((1, 2), (4, 10)), ((1, 2), (5, 6)), ((1, 2), (5, 7))\\
((1, 2), (5, 8)), ((1, 2), (5, 9)), ((1, 2), (5, 10)), ((1, 2), (5, 11)), ((1, 2), (6, 10)), ((1, 2), (8, 10)), ((1, 2), (9, 10))\\
((1, 2), (10, 11)), ((1, 3), (2, 7)), ((1, 3), (2, 9)), ((1, 3), (4, 9)), ((1, 3), (5, 9)), ((1, 3), (6, 7)), ((1, 3), (6, 9))\\
((1, 3), (7, 8)), ((1, 3), (7, 9)), ((1, 3), (7, 11)), ((1, 3), (8, 9)), ((1, 3), (9, 11)), ((1, 4), (2, 5)), ((1, 4), (2, 8))\\
((1, 4), (2, 11)), ((1, 4), (3, 9)), ((1, 4), (3, 11)), ((1, 4), (5, 6)), ((1, 4), (5, 8)), ((1, 4), (5, 9)), ((1, 4), (5, 11))\\
((1, 4), (6, 8)), ((1, 4), (6, 11)), ((1, 4), (8, 9)), ((1, 4), (8, 11)), ((1, 4), (9, 11)), ((1, 5), (2, 3)), ((1, 5), (2, 4))\\
((1, 5), (2, 6)), ((1, 5), (2, 8)), ((1, 5), (2, 10)), ((1, 5), (2, 11)), ((1, 5), (3, 6)), ((1, 5), (3, 8)), ((1, 5), (3, 9))\\
((1, 5), (3, 11)), ((1, 5), (4, 6)), ((1, 5), (4, 8)), ((1, 5), (4, 11)), ((1, 6), (2, 3)), ((1, 6), (2, 5)), ((1, 6), (2, 10))\\
((1, 6), (3, 5)), ((1, 6), (3, 7)), ((1, 6), (3, 9)), ((1, 6), (4, 5)), ((1, 6), (4, 10)), ((1, 6), (5, 8)), ((1, 6), (5, 9))\\
((1, 6), (5, 10)), ((1, 6), (5, 11)), ((1, 6), (8, 10)), ((1, 6), (9, 10)), ((1, 6), (10, 11)), ((1, 7), (2, 3)), ((1, 7), (2, 4))\\
((1, 7), (2, 6)), ((1, 7), (2, 8)), ((1, 7), (2, 11)), ((1, 7), (3, 4)), ((1, 7), (3, 6)), ((1, 7), (3, 8)), ((1, 7), (3, 9))\\
((1, 7), (3, 11)), ((1, 7), (4, 6)), ((1, 7), (4, 8)), ((1, 7), (4, 11)), ((1, 7), (6, 8)), ((1, 7), (6, 11)), ((1, 7), (8, 11))\\
((1, 8), (2, 5)), ((1, 8), (2, 10)), ((1, 8), (3, 5)), ((1, 8), (3, 6)), ((1, 8), (3, 9)), ((1, 8), (4, 5)), ((1, 8), (5, 6))\\
((1, 8), (5, 9)), ((1, 8), (5, 10)), ((1, 8), (5, 11)), ((1, 8), (6, 10)), ((1, 8), (9, 10)), ((1, 8), (10, 11)), ((1, 9), (2, 3))\\
((1, 9), (2, 10)), ((1, 9), (3, 4)), ((1, 9), (3, 5)), ((1, 9), (3, 6)), ((1, 9), (3, 7)), ((1, 9), (3, 8)), ((1, 9), (3, 10))\\
((1, 9), (3, 11)), ((1, 9), (4, 10)), ((1, 9), (5, 8)), ((1, 9), (5, 11)), ((1, 9), (6, 10)), ((1, 9), (8, 10)), ((1, 9), (10, 11))\\
((1, 10), (2, 3)), ((1, 10), (2, 5)), ((1, 10), (2, 6)), ((1, 10), (2, 8)), ((1, 10), (2, 9)), ((1, 10), (2, 11)), ((1, 10), (3, 5))\\
((1, 10), (3, 6)), ((1, 10), (3, 8)), ((1, 10), (3, 9)), ((1, 10), (5, 6)), ((1, 10), (5, 8)), ((1, 10), (5, 9)), ((1, 10), (5, 11))\\
((1, 10), (6, 8)), ((1, 10), (6, 9)), ((1, 11), (2, 5)), ((1, 11), (2, 10)), ((1, 11), (3, 5)), ((1, 11), (3, 9)), ((1, 11), (4, 5))\\
((1, 11), (5, 6)), ((1, 11), (5, 7)), ((1, 11), (5, 8)), ((1, 11), (5, 9)), ((1, 11), (5, 10)), ((1, 11), (6, 10)), ((1, 11), (8, 10))\\
((1, 11), (9, 10)), ((2, 3), (4, 7)), ((2, 3), (4, 9)), ((2, 3), (5, 7)), ((2, 3), (6, 7)), ((2, 3), (7, 8)), ((2, 3), (7, 9))\\
((2, 3), (7, 10)), ((2, 3), (7, 11)), ((2, 3), (8, 9)), ((2, 4), (5, 6)), ((2, 4), (5, 7)), ((2, 4), (5, 8)), ((2, 4), (5, 9))\\
((2, 4), (5, 11)), ((2, 4), (6, 7)), ((2, 4), (6, 9)), ((2, 4), (7, 8)), ((2, 4), (7, 9)), ((2, 4), (7, 11)), ((2, 4), (8, 9))\\
((2, 4), (9, 11)), ((2, 5), (3, 4)), ((2, 5), (4, 6)), ((2, 5), (4, 7)), ((2, 5), (4, 8)), ((2, 5), (4, 9)), ((2, 5), (4, 11))\\
((2, 5), (7, 10)), ((2, 6), (3, 7)), ((2, 6), (3, 10)), ((2, 6), (4, 7)), ((2, 6), (4, 10)), ((2, 6), (5, 7)), ((2, 6), (5, 10))\\
((2, 6), (7, 8)), ((2, 6), (7, 9)), ((2, 6), (7, 10)), ((2, 6), (7, 11)), ((2, 6), (8, 10)), ((2, 6), (9, 10)), ((2, 6), (10, 11))\\
((2, 7), (3, 4)), ((2, 7), (3, 5)), ((2, 7), (3, 6)), ((2, 7), (3, 8)), ((2, 7), (3, 9)), ((2, 7), (3, 11)), ((2, 7), (4, 6))\\
((2, 7), (4, 8)), ((2, 7), (4, 11)), ((2, 7), (6, 8)), ((2, 7), (6, 11)), ((2, 7), (8, 11)), ((2, 8), (3, 10)), ((2, 8), (4, 5))\\
((2, 8), (4, 7)), ((2, 8), (4, 10)), ((2, 8), (5, 7)), ((2, 8), (5, 10)), ((2, 8), (6, 7)), ((2, 8), (6, 10)), ((2, 8), (7, 10))\\
((2, 8), (7, 11)), ((2, 8), (9, 10)), ((2, 8), (10, 11)), ((2, 9), (3, 4)), ((2, 9), (3, 10)), ((2, 9), (4, 6)), ((2, 9), (4, 7))\\
((2, 9), (4, 8)), ((2, 9), (4, 10)), ((2, 9), (4, 11)), ((2, 9), (5, 10)), ((2, 9), (6, 10)), ((2, 9), (7, 10)), ((2, 9), (8, 10))\\
((2, 9), (10, 11)), ((2, 10), (4, 9)), ((2, 10), (6, 8)), ((2, 10), (6, 9)), ((2, 10), (6, 11)), ((2, 10), (8, 9)), ((2, 10), (8, 11))\\
((2, 10), (9, 11)), ((2, 11), (3, 10)), ((2, 11), (4, 7)), ((2, 11), (4, 10)), ((2, 11), (5, 7)), ((2, 11), (5, 10)), ((2, 11), (6, 7))\\
((2, 11), (6, 10)), ((2, 11), (7, 8)), ((2, 11), (7, 9)), ((2, 11), (7, 10)), ((2, 11), (8, 10)), ((2, 11), (9, 10)), ((3, 4), (5, 7))\\
((3, 4), (5, 9)), ((3, 4), (6, 7)), ((3, 4), (6, 9)), ((3, 4), (7, 8)), ((3, 4), (7, 9)), ((3, 4), (7, 11)), ((3, 4), (8, 9))\\
((3, 4), (9, 11)), ((3, 5), (4, 7)), ((3, 5), (4, 9)), ((3, 5), (6, 7)), ((3, 5), (7, 8)), ((3, 5), (7, 9)), ((3, 5), (7, 10))\\
((3, 5), (7, 11)), ((3, 6), (4, 7)), ((3, 6), (4, 9)), ((3, 6), (5, 7)), ((3, 6), (7, 8)), ((3, 6), (7, 9)), ((3, 6), (7, 10))\\
((3, 6), (7, 11)), ((3, 7), (4, 6)), ((3, 7), (4, 9)), ((3, 7), (5, 6)), ((3, 7), (6, 8)), ((3, 7), (6, 9)), ((3, 7), (6, 11))\\
((3, 7), (8, 11)), ((3, 8), (4, 5)), ((3, 8), (4, 7)), ((3, 8), (4, 9)), ((3, 8), (5, 7)), ((3, 8), (6, 7)), ((3, 8), (7, 9))\\
((3, 8), (7, 10)), ((3, 8), (7, 11)), ((3, 9), (6, 7)), ((3, 9), (7, 8)), ((3, 9), (7, 11)), ((3, 10), (4, 9)), ((3, 10), (6, 8))\\
((3, 10), (6, 9)), ((3, 10), (6, 11)), ((3, 10), (8, 9)), ((3, 10), (8, 11)), ((3, 10), (9, 11)), ((3, 11), (4, 5)), ((3, 11), (4, 7))\\
((3, 11), (4, 9)), ((3, 11), (5, 7)), ((3, 11), (6, 7)), ((3, 11), (6, 9)), ((3, 11), (7, 8)), ((3, 11), (7, 9)), ((3, 11), (7, 10))\\
((4, 5), (6, 8)), ((4, 5), (8, 9)), ((4, 5), (8, 11)), ((4, 5), (9, 11)), ((4, 6), (5, 7)), ((4, 6), (5, 8)), ((4, 6), (5, 9))\\
((4, 6), (5, 11)), ((4, 6), (7, 8)), ((4, 6), (7, 9)), ((4, 6), (7, 11)), ((4, 6), (8, 9)), ((4, 6), (9, 11)), ((4, 7), (5, 6))\\
((4, 7), (5, 8)), ((4, 7), (5, 9)), ((4, 7), (5, 11)), ((4, 7), (6, 11)), ((4, 7), (8, 11)), ((4, 7), (9, 11)), ((4, 8), (5, 6))\\
((4, 8), (5, 7)), ((4, 8), (5, 9)), ((4, 8), (5, 11)), ((4, 8), (6, 7)), ((4, 8), (6, 9)), ((4, 8), (7, 9)), ((4, 8), (7, 11))\\
((4, 8), (9, 11)), ((4, 9), (5, 8)), ((4, 9), (5, 11)), ((4, 9), (6, 10)), ((4, 9), (8, 10)), ((4, 9), (10, 11)), ((4, 10), (6, 8))\\
((4, 10), (6, 9)), ((4, 10), (6, 11)), ((4, 10), (8, 9)), ((4, 10), (8, 11)), ((4, 10), (9, 11)), ((4, 11), (5, 6)), ((4, 11), (5, 7))\\
((4, 11), (5, 8)), ((4, 11), (5, 9)), ((4, 11), (6, 7)), ((4, 11), (6, 9)), ((4, 11), (7, 8)), ((4, 11), (7, 9)), ((4, 11), (8, 9))\\
((5, 7), (6, 11)), ((5, 7), (8, 11)), ((5, 10), (6, 8)), ((5, 10), (6, 9)), ((5, 10), (6, 11)), ((5, 10), (8, 9)), ((5, 10), (8, 11))\\
((5, 10), (9, 11)), ((6, 7), (8, 11)), ((6, 8), (7, 10)), ((6, 8), (7, 11)), ((6, 8), (9, 10)), ((6, 8), (10, 11)), ((6, 9), (8, 10))\\
((6, 9), (10, 11)), ((6, 10), (8, 9)), ((6, 10), (8, 11)), ((6, 10), (9, 11)), ((6, 11), (7, 8)), ((6, 11), (7, 9)), ((6, 11), (7, 10))\\
((6, 11), (8, 10)), ((6, 11), (9, 10)), ((7, 9), (8, 11)), ((7, 10), (8, 11)), ((8, 9), (10, 11)), ((8, 10), (9, 11)), ((8, 11), (9, 10))$

The rest of the edges are blue. 

\section{Appendix: Our code}\label{Apen: code}
Below is the code we used to obtain the results in Theorem~\ref{thm: computational}.

\begin{verbatim}
    import itertools as it
    from pysat.solvers import Glucose42
    
    '''
    Builds Kneser Graph Kn(n, r)
    Input: n, r
    Output: (V, E)
        V (list): Vertices (all r-subsets of [n]) as r-tuples
        E (list of 2-tuples): Edges as a pair of r-tuples
    '''
    def build_kneser_graph(n,r):
        V = list(it.combinations(range(1, n + 1), r))
        E = []
        for (a, b) in it.combinations(V, 2):
            if set(a).isdisjoint(set(b)):
                E.append((a, b))
                
        return (V, E)
    
    '''
    Try to find if Kn(n,r) can be 2-colored so that there is no red K_s or blue K_t
    
    For no red K_s, for every s-clique, need at least 1 blue edge. 
    That is, if e_1, ..., e_k are the edges of an s-clique and e_i = 1 if red, 
    e_i = -1 if blue, want to satisfy (-e_1 v -e_2 v ... v -e_k)
    
    For no blue K_t, for every t-clique, need at least 1 red edge. 
    That is, if e_1, ..., e_k are the edges of a t-clique and e_i = 1 if red, 
    e_i = -1 if blue, want to satisfy (-e_1 v -e_2 v ... v -e_k)
    
    (NOTE: 1 corresponds to True, -1 corresponds to False)
        
    If there is a good coloring, outputs (True, R, B), where R = red edges, 
    B = blue edges
    If there is no good coloring, outputs (False, [], [])
    '''
    def build_sat_instances(n, r, s, t):
        solver = Glucose42()
        clauses = []
        V = []
        E = []
        
        #Create vertex and edge list
        (V, E) = build_kneser_graph(n, r)
        
    
        for s_tuple in it.combinations(V, s):
            
            #check that the s-tuple forms a clique in Kn(n,r)
            if all([set(s_tuple[i]).isdisjoint(set(s_tuple[j])) for (i,j) in 
            it.combinations(range(len(s_tuple)), 2)]):
                
                #Add clause forbidding monochromatic K_s and K_t (if s = t)
                if s == t:
                    cliq = []
                    for edg in it.combinations(s_tuple, 2):
                        idx = -1
                        try:
                            idx = E.index(edg)
                            
                        except:
                            idx = E.index((edg[1], edg[0]))
                            
                        cliq.append(idx + 1)
                    
                    solver.add_clause([-x for x in cliq])
                    solver.add_clause(cliq)
                   
                #Add clause forbidding red K_s 
                if s != t:
                    cliq = []
                    for edg in it.combinations(s_tuple, 2):
                        idx = -1
                        try:
                            idx = E.index(edg)
                            
                        except:
                            idx = E.index((edg[1], edg[0]))
                            
                        cliq.append(idx + 1)
                            
                    solver.add_clause([-x for x in cliq])
                        
        if s != t:
            for t_tuple in it.combinations(V, t):
                
                #check that the t-tuple forms a clique in Kn(n,r)
                if all([set(t_tuple[i]).isdisjoint(set(t_tuple[j])) for (i,j) in
                it.combinations(range(len(t_tuple)), 2)]):
                    
                    #Add clause forbidding red K_t
                    cliq = []
                    for edg in it.combinations(t_tuple, 2):
                        idx = -1
                        try:
                            idx = E.index(edg)
                            
                        except:
                            idx = E.index((edg[1], edg[0]))
                            
                        cliq.append(idx + 1)
                            
                    solver.add_clause(cliq)
        
        #Solve SAT problem and recover coloring
        result = solver.solve()
        R = []
        B = []
        if result:
            L = solver.get_model()
            R = [E[i - 1] for i in L if i > 0]
            B = [E[-i - 1] for i in L if i < 0]
        
        return (result, R, B)
    
    #--------------------------------------------------
    #output: True, R, B
    build_sat_instances(8, 2, 3, 3)
    
    #output: False, [], []
    build_sat_instances(9, 2, 3, 3)
    
    #--------------------------------------------------
    
    #output: False, [], []
    build_sat_instances(13, 3, 3, 3)

\end{verbatim}

\end{document}